\documentclass[11pt,a4paper,reqno]{amsart}
\usepackage{amsfonts}
\usepackage{amsmath}
\usepackage{amssymb}
\usepackage{hyperref}

\usepackage{amsthm}
\usepackage{enumitem}

\usepackage{dutchcal}

\newcommand{\cE}{\mathcal E}

\newcommand{\N}{\mathbb N}
\newcommand{\1}{\mathbf 1}
\newcommand{\dd}{\mathrm d}

\newcommand{\eps}{\varepsilon}
\numberwithin{equation}{section}
\newtheorem{theorem}{Theorem}[section]
\newtheorem{lemma}[theorem]{Lemma}
\newtheorem{remark}[theorem]{Remark}
\newtheorem{corollary}[theorem]{Corollary}
\newtheorem{proposition}[theorem]{Proposition}

\allowdisplaybreaks

\def\enddoc{\end{document}}

\begin{document}
\author{Qingsong Gu}
\address{Department of mathematics, Nanjing University, Nanjing 210093, P. R. China} \email{qingsonggu@nju.edu.cn}

\author{Lu Hao}
\address{Universit\"{a}t Bielefeld, Fakult\"{a}t f\"{u}r Mathematik, Postfach 100131, D-33501, Bielefeld, Germany}
\email{lhao@math.uni-bielefeld.de}

\author{Xueping Huang}
\address{School of Mathematics and Statistics, Nanjing University of Information Science and Technology,
	Nanjing 210044, P. R. China}
\email{hxp@nuist.edu.cn}

\author{Yuhua Sun}
\address{School of Mathematical Sciences and LPMC, Nankai University, 300071
	Tianjin, P. R. China}
\email{sunyuhua@nankai.edu.cn}
\title[Lane--Emden system]
{A sharp integral criterion for the Lane--Emden system of inequalities
on weighted graphs}

\thanks{\noindent
	Q. Gu was supported by the National Natural Science Foundation of China (Grant Nos. 12101303 and 12171354).
	L. Hao was funded by the Deutsche Forschungsgemeinschaft (DFG, German Research Foundation), Project-ID 317210226, SFB 1283.
	X. Huang was supported by
	the National Natural Science Foundation of China (Grant No. 11601238).
	Y. Sun was funded by the National Natural Science Foundation of
	China (Grant No. 12371206) and the Fundamental Research Funds for the Central Universities, No. 050-63263078.}

\begin{abstract}
	We establish a sharp integral nonexistence criterion for the Lane--Emden
	system of inequalities
	\[
	-\Delta u\ge v^p,\qquad -\Delta v\ge u^q,
	\qquad p,q>0,\quad pq>1,
	\]
	on arbitrary infinite, connected, locally finite weighted graphs.  In the
	asymmetric case $p\ne q$, set $P=\max\{p,q\}$.  If, for some root $o\in V$,
	\[
	\sum_{n=2}^{\infty}
	\frac{n^{2pq+2P-1}}{\mu(B(o,n))^{pq-1}}=\infty,
	\]
	then every nonnegative solution $(u,v)$ satisfies $u\equiv v\equiv0$.
	The proof combines flow decomposition of the finite Green current with
	nonlinear testing.  In the symmetric case $p=q>1$, the Liouville problem
	reduces, via the sum $u+v$, to the scalar criterion
	\[
	\sum_{n=2}^{\infty}
	\frac{n^{2p-1}}{\mu(B(o,n))^{p-1}}=\infty.
	\]
	Weighted half-line examples show that the critical logarithmic endpoint
	in the asymmetric result is sharp.
\end{abstract}


\maketitle

\section{Introduction}

Liouville-type theorems for nonlinear elliptic equations and inequalities
reflect the interaction between the nonlinearity and the geometry of the
underlying space.  Fundamental examples are the Lane--Emden equation
\[
-\Delta u=u^q,\qquad q>1,
\]
and its coupled counterpart
\[
-\Delta u=v^p,\qquad -\Delta v=u^q,
\qquad p,q>0,\quad pq>1.
\]
Although the scalar problem has been studied extensively, the coupled
system is substantially more delicate because the two components interact
and may have different scaling orders.  Even in Euclidean space, optimal
exponents and borderline phenomena for Lane--Emden systems require more
refined arguments than in the scalar case; see
\cite{BuscaManasevich02,GidasSpruck81,Mitidieri96,MitidieriPokhozhaev01,
	PQS07,SerrinZou96,Souplet09}.

A central principle in geometric analysis is that Euclidean dimension can
be replaced by geometric information at infinity.  On complete noncompact
Riemannian manifolds, volume growth plays an essential role in parabolicity,
stochastic completeness, and Liouville properties.  Classical criteria
connect the growth of geodesic balls with recurrence of Brownian motion and
stochastic completeness; see
\cite{GrigorYannParabolic,GrigorYannStoch}.  This viewpoint has led to many
Liouville-type results for nonlinear elliptic inequalities on manifolds, in
which volume growth replaces dimension; see
\cite{GS14,GSV20,Sun16,XuWangSun18}.

The same philosophy extends naturally to weighted graphs.  Let
$(V,E,\mu)$ be an infinite, connected, locally finite weighted graph.  Here
$V$ is the vertex set, $E$ is the edge set, and
$\mu=(\mu_{xy})_{x\sim y}$ is a family of symmetric positive edge
weights, so that $\mu_{xy}=\mu_{yx}>0$ whenever $x\sim y$.  We write
$x\sim y$ when $\{x,y\}\in E$, and set
\[
\mu(x)=\sum_{y\sim x}\mu_{xy}<\infty.
\]
For $A\subset V$, write
\[
\mu(A)=\sum_{x\in A}\mu(x).
\]
For a real-valued function $u$ on $V$, define the normalized graph
Laplacian by
\[
\Delta u(x)=\frac{1}{\mu(x)}
\sum_{y\sim x}\mu_{xy}\bigl(u(y)-u(x)\bigr).
\]
We write $d(x,y)$ for the combinatorial graph distance on $(V,E)$.  Fix a
root $o\in V$, and set
\[
B_n=B(o,n)=\{x\in V:d(o,x)\le n\}.
\]
%
%

Infinite weighted graphs provide a discrete setting in which volume
growth, recurrence, and potential theory are closely related; see
\cite{AGbook,WWbook} for general background.  As one illustration, the
Nash--Williams criterion yields the following volume-growth sufficient
condition for recurrence: if
\[
\sum_{n=1}^{\infty}\frac{n}{\mu(B_n)}=\infty,
\]
then every nonnegative function satisfying $-\Delta u\ge0$ is constant; see
\cite{LyonsPeres17,NashWilliams59,Soardi94}.  Related inequalities for
$p$-superharmonic functions on networks and for superharmonic functions on
graphs were developed in \cite{SC95} and \cite{SC97}, respectively.  More
recently, scalar semilinear inequalities and nonlinear extensions on
weighted graphs have been investigated in
\cite{DLS26,GW25,GHS23,HS23}.

The aim of the present paper is to establish a sharp integral nonexistence
criterion for nonnegative solutions of the Lane--Emden system of
inequalities
\begin{equation}\label{eq:system-normalized}
	-\Delta u\ge v^p,
	\qquad
	-\Delta v\ge u^q,
	\qquad p,q>0,\quad pq>1,
\end{equation}
on arbitrary locally finite weighted graphs.  A pair $(u,v)$ is called a nonnegative solution of
\eqref{eq:system-normalized} if $u,v\ge0$ and both inequalities hold
pointwise on $V$.

Several Liouville-type theorems for Lane--Emden systems on weighted graphs
have recently been established in
\cite{DuongMinh25,MinhDuongNguyen24,MinhQuyetDuong25}.  These works
underscore the additional difficulties caused by coupling.  The scalar
problem provides a useful benchmark: the authors recently proved that
\[
\sum_{n=1}^{\infty}
\frac{n^{2q-1}}{\mu(B_n)^{q-1}}=\infty
\]
forces every nonnegative solution of $-\Delta u\ge u^q$ to vanish
identically; see \cite{GHHS26}.
This integral criterion is not restricted
to a prescribed asymptotic form of the volume.  In particular, within the
polynomial--logarithmic scale it identifies the critical exponents
$2q/(q-1)$ and $1/(q-1)$.

For the coupled system, standard test-function arguments identify the
critical polynomial scale
\[
R^{\,2+2(P+1)/(pq-1)},\qquad P=\max\{p,q\},
\]
but do not determine the logarithmic endpoint at that scale.  Resolving
this loss of information is the main issue addressed here.  We obtain an
integral criterion for the asymmetric system that is valid on arbitrary
locally finite weighted graphs and requires neither volume doubling nor a
Poincar\'{e} inequality.

We now state the results in the asymmetric case $p\ne q$ and the symmetric
case $p=q$.
\begin{theorem}\label{thm:sharp-asymmetric}
	Assume
	\[
	p,q>0,
	\qquad pq>1,
	\qquad p\ne q,
	\]
	and set
	\[
	P=\max\{p,q\}.
	\]
	If, for one root $o\in V$,
	\begin{equation}\label{eq:sharp-series}
		\sum_{n=2}^{\infty}
		\frac{n^{2pq+2P-1}}{\mu(B_n)^{pq-1}}
		=\infty,
	\end{equation}
	then every nonnegative solution $(u,v)$ of
	\eqref{eq:system-normalized} satisfies $u\equiv v\equiv0$.
\end{theorem}

\begin{remark}
	The divergence of \eqref{eq:sharp-series} is independent of the chosen
	root.  If $d(o,o')=m$, then
	\[
	B(o,n-m)\subset B(o',n)\subset B(o,n+m),
	\qquad n\ge m,
	\]
	and hence
	\[
	\mu(B(o,n-m))\le \mu(B(o',n))\le \mu(B(o,n+m)).
	\]
	After reindexing, the factors $(n\pm m)^{2pq+2P-1}$ are comparable to
	$n^{2pq+2P-1}$ for all sufficiently large $n$.  Hence the two positive
	series associated with $o$ and $o'$ converge or diverge together.
\end{remark}

\begin{corollary}\label{cor:critical-log}
	Let $p\ne q$, $pq>1$, and $P=\max\{p,q\}$.  Define the critical volume exponent
	\begin{align}
		D_{\rm sys}=2+\frac{2(P+1)}{pq-1}.
	\end{align}
	If, for all sufficiently large $R$,
	\[
	\mu(B_R)\le C R^{D_{\rm sys}}(\log R)^\theta
	\]
	and
	\[
	\theta\le \frac{1}{pq-1},
	\]
	then every nonnegative solution $(u,v)$ of
	\eqref{eq:system-normalized} satisfies $u\equiv v\equiv0$.
\end{corollary}
\begin{remark}
	Earlier Liouville theorems for elliptic systems on weighted graphs were
	proved in
	\cite{DuongMinh25,MinhDuongNguyen24,MinhQuyetDuong25}.
	The work \cite{DuongMinh25} treats a standard-Laplacian Lane--Emden
	system under additional assumptions on the edge weights and volume
	growth.  The work \cite{MinhDuongNguyen24} allows positive potentials in
	the two equations, under hypotheses involving a suitable graph distance
	and weighted volume growth, while \cite{MinhQuyetDuong25} treats systems
	involving possibly different graph $p_i$-Laplacians and a coefficient in
	the nonlinear terms.  The present results are more specialized with
	respect to the operators and coefficients.  For the constant-coefficient
	system \eqref{eq:system-normalized}, however, Theorem
	\ref{thm:sharp-asymmetric} gives an integral criterion on arbitrary
	locally finite weighted graphs, and Corollary
	\ref{cor:critical-log} identifies the critical logarithmic endpoint in
	the polynomial--logarithmic scale.  Thus these results are complementary.
\end{remark}

\begin{theorem}\label{thm:symmetric}
	Let $p=q=\sigma>1$.  If
	\begin{equation}\label{eq:scalar-series}
		\sum_{n=2}^{\infty}
		\frac{n^{2\sigma-1}}{\mu(B_n)^{\sigma-1}}=\infty,
	\end{equation}
	then every nonnegative solution $(u,v)$ of
	\[
	-\Delta u\ge v^\sigma,
	\qquad
	-\Delta v\ge u^\sigma
	\]
	satisfies $u\equiv v\equiv0$.
\end{theorem}

A striking feature appears at the critical logarithmic scale.  The
polynomial exponent is continuous across the diagonal: as
$(p,q)\to(\sigma,\sigma)$ through $p\ne q$, one has
\[
D_{\rm sys}\longrightarrow
2+\frac{2(\sigma+1)}{\sigma^2-1}
=\frac{2\sigma}{\sigma-1},
\]
which is the scalar critical exponent.  The logarithmic endpoint is \emph{not}
continuous.  The asymmetric value $1/(pq-1)$ tends to
$1/(\sigma^2-1)$, whereas on the diagonal the Liouville question reduces,
via $w=u+v$, to the scalar inequality and the endpoint is
$1/(\sigma-1)$.  Thus the symmetric logarithmic threshold cannot be
recovered by taking a limit of the asymmetric formula.

Theorem \ref{thm:sharp-asymmetric} is formulated in terms of the full volume
sequence and therefore has a wider scope than the
polynomial--logarithmic corollary.  Its critical endpoint is optimal in the
following precise sense.  For every $\eps>0$, Section \ref{counter}
constructs a weighted half-line satisfying
\[
\mu(B(o,R))\asymp
R^{D_{\rm sys}}(\log R)^{1/(pq-1)+\eps}
\]
on which \eqref{eq:system-normalized} admits a positive solution.  


The key new ingredient is a weighted lower estimate obtained from a flow
decomposition of the finite Green current.  The flow decomposition theorem
for finite acyclic flows is classical; see \cite{FFbook}.  It represents
the current as a probability measure on directed root-to-boundary paths, along
which one-dimensional Hardy estimates can be applied.  Averaging these
estimates over the decomposition and using a parallel-sum argument lead to
the series in \eqref{eq:sharp-series}.  This mechanism first appeared in our recent work \cite{GHHS26} for scalar Lane--Emden inequalities. For related background on flows in
networks, see \cite{BermanKonsowa90,LyonsPeres17,Soardi94}.  The new point
is to use flow decomposition as an analytic device for the nonlinear
system.  Combined with the upper estimate obtained by Picone testing, it
connects volume growth with the weighted Green mass and determines the
sharp threshold for the coupled system.

The paper is organized as follows.  Section \ref{conse} introduces the
finite Green voltage and establishes the preliminary estimates.  Section
\ref{upp-bound} proves a uniform upper bound for the weighted Green mass in
the presence of a nontrivial nonnegative solution of the asymmetric
system.  Section \ref{lbdd} derives the complementary lower bound by flow
decomposition.  Section \ref{prf} contains the proofs of Theorem
\ref{thm:sharp-asymmetric}, Corollary \ref{cor:critical-log}, and Theorem
\ref{thm:symmetric}.  Finally, Section \ref{counter} establishes the
sharpness of the logarithmic endpoint in Corollary
\ref{cor:critical-log} within the critical polynomial--logarithmic volume
scale.

\section{Preliminary estimates}\label{conse}
Throughout the proof sections, all radius parameters are positive integers.
For the proofs it is convenient to use the operator
\[
Lf(x):=\mu(x)(-\Delta f)(x)
=\sum_{y\sim x}\mu_{xy}\bigl(f(x)-f(y)\bigr).
\]
Thus \eqref{eq:system-normalized} is equivalently written as
\[
Lu(x)\ge \mu(x)v(x)^p,
\qquad
Lv(x)\ge \mu(x)u(x)^q,
\qquad x\in V.
\]
We use $L$ only in the proof calculations; the system and the main results
remain formulated in terms of the normalized Laplacian $\Delta$.

For every $R\ge1$, let $g_R$ be the unique function satisfying
\begin{equation}\label{eq:green-equation}
	Lg_R(x)=\1_{\{o\}}(x),\qquad x\in B_R,
\end{equation}
and $g_R=0$ on $V\setminus B_R$.  Since $B_R$ is finite and connected and
has nonempty edge boundary, the finite Dirichlet problem has a unique
solution; the discrete strong maximum principle gives $g_R>0$ on $B_R$.

For any real-valued function $f$ on $V$ and every finitely supported function $\varphi$, the
following edge sum is finite, and we write
\[
\cE(f,\varphi)
=\sum_{\{x,y\}\in E}\mu_{xy}\bigl(f(x)-f(y)\bigr)
\bigl(\varphi(x)-\varphi(y)\bigr),
\]
where the sum is over unordered edges.  Summation by parts gives
\[
\cE(f,\varphi)
=\sum_{x\in V}Lf(x)\varphi(x)
=\sum_{x\in V}(-\Delta f)(x)\varphi(x)\mu(x).
\]
In particular, from \eqref{eq:green-equation},
\begin{equation}\label{eq:green-energy-identity}
	\cE(g_R,\varphi)=\varphi(o)
\end{equation}
for every function $\varphi$ supported in $B_R$.

Set the shifted radius
\[
\rho(x)=1+d(o,x),\qquad x\in V.
\]
For $s>1$ and $P\ge0$, define the weighted Green mass by
\begin{equation*}
	M_{s,P}(R)
	=\sum_{x\in B_R}\rho(x)^{2P}g_R(x)^s\mu(x).
\end{equation*}
Here $P$ is a generic nonnegative weight exponent.  In the asymmetric
Lane--Emden argument it will be specialized to $P=\max\{p,q\}$.

Later, we show that, in the asymmetric case, the existence of a
nontrivial nonnegative solution of \eqref{eq:system-normalized} forces
$M_{pq,P}(R)$ to be uniformly bounded in $R$; see Proposition
\ref{prop:system-upper-endpoint}.  The Green-current flow gives the
corresponding lower bound; see Proposition
\ref{prop:weighted-flow-lower}.  We first record several preliminary
consequences and estimates for \eqref{eq:system-normalized}.
\begin{lemma}\label{lem:zero-propagation}
	Let $(u,v)$ be a nonnegative solution of \eqref{eq:system-normalized}.  If
	$u(x_0)=0$ or $v(x_0)=0$ at one vertex $x_0$, then $u\equiv v\equiv0$ on $V$.  Hence every
	nontrivial nonnegative solution is strictly positive everywhere.
\end{lemma}

\begin{proof}
	Assume $u(x_0)=0$.  Since $u\ge0$,
	\[
	Lu(x_0)=\sum_{y\sim x_0}\mu_{x_0y}(0-u(y))\le0.
	\]
	But the $u$-inequality in \eqref{eq:system-normalized} gives
	\[
	Lu(x_0)\ge v(x_0)^p\mu(x_0)\ge0.
	\]
	Hence $v(x_0)=0$ and $u(y)=0$ for every $y\sim x_0$.  Applying the $v$-inequality at
	$x_0$ gives
	\[
	Lv(x_0)=\sum_{y\sim x_0}\mu_{x_0y}(0-v(y))
	\ge u(x_0)^q\mu(x_0)=0,
	\]
	while the left side is nonpositive.  Therefore $v(y)=0$ for every $y\sim x_0$.  Repeating
	this argument along paths and using connectedness gives $u\equiv v\equiv0$.  The case
	$v(x_0)=0$ is symmetric.  This completes the proof.
\end{proof}
We next record an elementary Picone inequality.
\begin{lemma}\label{lem:picone}
	Let $a,b>0$, $s,t\ge0$, and $\alpha>1$.  Then
	\begin{equation*}
		(a-b)(s^\alpha-t^\alpha)
		\le
		\frac{\alpha}{\alpha-1}(as-bt)(s^{\alpha-1}-t^{\alpha-1}).
	\end{equation*}
\end{lemma}

\begin{proof}
	If $s=t$, the claim is trivial.  By interchanging $(a,s)$ and $(b,t)$ if necessary, assume
	$s>t$.  For every $\lambda\in[t,s]$,
	\[
	as-bt-\lambda(a-b)=a(s-\lambda)+b(\lambda-t)\ge0.
	\]
	Thus $\lambda(a-b)\le as-bt$.  Multiplying by $\lambda^{\alpha-2}$ and integrating from
	$t$ to $s$, we obtain
	\[
	\begin{aligned}
		(a-b)(s^\alpha-t^\alpha)
		&=\alpha\int_t^s\lambda^{\alpha-1}(a-b)\,d\lambda  \\
		&\le \alpha(as-bt)\int_t^s\lambda^{\alpha-2}\,d\lambda  \\
		&=\frac{\alpha}{\alpha-1}(as-bt)(s^{\alpha-1}-t^{\alpha-1}).
	\end{aligned}
	\]
\end{proof}

\begin{lemma}
	Assume that \eqref{eq:system-normalized} has a nontrivial nonnegative solution
	$(u,v)$.  Let $A>1$ and $B>1$.  Then, for every $R\ge1$,
	\begin{align}
		\sum_{x\in B_R}
		v(x)^p\left(\frac{g_R(x)}{u(x)}\right)^A\mu(x)
		&\le
		\frac{A}{A-1}\left(\frac{g_R(o)}{u(o)}\right)^{A-1},
		\label{eq:test-u}\\
		\sum_{x\in B_R}
		u(x)^q\left(\frac{g_R(x)}{v(x)}\right)^B\mu(x)
		&\le
		\frac{B}{B-1}\left(\frac{g_R(o)}{v(o)}\right)^{B-1}.
		\label{eq:test-v}
	\end{align}
\end{lemma}

\begin{proof}
	By Lemma \ref{lem:zero-propagation}, $u,v>0$ everywhere.  Set
	$w=g_R/u$.  Since $g_R$ is finitely supported, so is $w$.  Testing
	$Lu\ge v^p\mu$ against $w^A$ gives
	\[
	\sum_{x\in B_R}v(x)^p w(x)^A\mu(x)
	\le \sum_{x\in V}Lu(x)w(x)^A
	=\cE(u,w^A).
	\]
	Apply Lemma \ref{lem:picone} on each edge $\{x,y\}$ with
	\[
	a=u(x),\quad b=u(y),\quad s=w(x),\quad t=w(y).
	\]
	Since $as=g_R(x)$ and $bt=g_R(y)$, summing over unordered edges yields
	\[
	\cE(u,w^A)
	\le
	\frac{A}{A-1}\cE(g_R,w^{A-1}).
	\]
	By \eqref{eq:green-energy-identity},
	\[
	\cE(g_R,w^{A-1})=w(o)^{A-1}
	=\left(\frac{g_R(o)}{u(o)}\right)^{A-1}.
	\]
	This proves \eqref{eq:test-u}.  The proof of \eqref{eq:test-v} is identical, testing
	$Lv\ge u^q\mu$ with $(g_R/v)^B$.
\end{proof}



We next establish the weighted Green estimate used in the asymmetric
argument.  The first ingredient is the following ground-state inequality.
\begin{lemma}
	Let $h:V\to(0,\infty)$ and $\lambda:V\to(0,\infty)$ satisfy
	\begin{equation*}
		Lh(x)\ge \lambda(x)h(x)\mu(x),
		\qquad x\in V.
	\end{equation*}
	Then every finitely supported function $\phi$ satisfies
	\begin{equation}\label{eq:ground-state-hardy}
		\sum_{x\in V}\lambda(x)\phi(x)^2\mu(x)
		\le \cE(\phi,\phi).
	\end{equation}
\end{lemma}

\begin{proof}
	It is enough to prove
	\[
	\sum_{x\in V}\frac{Lh(x)}{h(x)}\phi(x)^2\le \cE(\phi,\phi).
	\]
	Set $r=\phi/h$.  Then
	\[
	\begin{aligned}
		\cE(\phi,\phi)
		&-\sum_{x\in V}\frac{Lh(x)}{h(x)}\phi(x)^2  \\
		&=\sum_{\{x,y\}\in E}\mu_{xy}h(x)h(y)
		\bigl(r(x)-r(y)\bigr)^2\ge0.
	\end{aligned}
	\]
	Since $Lh(x)/h(x)\ge \lambda(x)\mu(x)$ pointwise,
	\eqref{eq:ground-state-hardy} follows.
\end{proof}

\begin{lemma}\label{lem:two-point-caccioppoli}
	For each $s>1$, there are constants $c_s,C_s>0$ such that the following
	estimate holds whenever $a,b,\alpha,\beta\ge0$:
	\begin{equation}\label{eq:two-point-caccioppoli}
		\begin{aligned}
			&(a-b)\bigl(\alpha^2a^{s-1}-\beta^2b^{s-1}\bigr)  \\
			&\qquad\ge
			c_s\bigl(\alpha a^{s/2}-\beta b^{s/2}\bigr)^2
			-C_s(\alpha-\beta)^2(a^s+b^s).
		\end{aligned}
	\end{equation}
\end{lemma}

\begin{proof}
	Set
	\[
	M=(a-b)(a^{s-1}-b^{s-1}),\qquad
	Q=(a-b)(a^{s-1}+b^{s-1}),\qquad
	S=a^s+b^s.
	\]
	The power monotonicity estimate
	\begin{equation}\label{eq:power-monotonicity}
		M\ge c_s^0\bigl(a^{s/2}-b^{s/2}\bigr)^2,
		\qquad c_s^0=\frac{4(s-1)}{s^2},
	\end{equation}
	follows directly from Cauchy--Schwarz.  Indeed, after interchanging $a$
	and $b$ if necessary, assume $a\ge b$; then
	\[
	\bigl(a^{s/2}-b^{s/2}\bigr)^2
	=\frac{s^2}{4}\left(\int_b^a t^{s/2-1}\,dt\right)^2
	\le \frac{s^2}{4(s-1)}M.
	\]
	We also have
	\begin{equation}\label{eq:cross-estimate}
		|Q|\le K_s M^{1/2}S^{1/2}.
	\end{equation}
	To see this, again assume $a\ge b$.  If $a>0$, then, with $t=b/a\in[0,1]$,
	\[
	\begin{aligned}
	Q
	&=a^s(1-t)(1+t^{s-1})
	\le 2a^s(1-t)
	\le 2a^s(1-t^s)  \\
	&=2\bigl(a^{s/2}-b^{s/2}\bigr)
	\bigl(a^{s/2}+b^{s/2}\bigr).
	\end{aligned}
	\]
	The case $a=b=0$ is trivial. Hence \eqref{eq:cross-estimate} now follows
	from \eqref{eq:power-monotonicity} and
	$(a^{s/2}+b^{s/2})^2\le2S$.

	Let
	\[
	m=\frac{\alpha+\beta}{2},\qquad
	d=\frac{\alpha-\beta}{2}.
	\]
	A direct expansion gives
	\[
	(a-b)\bigl(\alpha^2a^{s-1}-\beta^2b^{s-1}\bigr)
	=(m^2+d^2)M+2mdQ.
	\]
	By \eqref{eq:cross-estimate} and Young's inequality,
	\[
	(m^2+d^2)M+2mdQ
	\ge \frac12m^2M-C_s d^2S.
	\]
	On the other hand,
	\begin{equation}\label{eq:split-square}
		\begin{aligned}
		\bigl(\alpha a^{s/2}-\beta b^{s/2}\bigr)^2
		&=\Bigl[m\bigl(a^{s/2}-b^{s/2}\bigr)
		+d\bigl(a^{s/2}+b^{s/2}\bigr)\Bigr]^2  \\
		&\le \frac{2}{c_s^0}m^2M+4d^2S.
		\end{aligned}
	\end{equation}
	In particular, \eqref{eq:split-square} implies
	\[
	\frac12m^2M\ge
	\frac{c_s^0}{4}\bigl(\alpha a^{s/2}-\beta b^{s/2}\bigr)^2
	-c_s^0d^2S.
	\]
	Combining this with the preceding lower bound and using
	$d^2=(\alpha-\beta)^2/4$ proves \eqref{eq:two-point-caccioppoli}.
\end{proof}

\begin{lemma}\label{lem:green-caccioppoli}
	Let $s>1$ and let $\psi:V\to[0,\infty)$.  Then
	\begin{equation}\label{eq:green-caccioppoli}
		\begin{aligned}
			\cE(\psi g_R^{s/2},\psi g_R^{s/2})
			&\le
			C_s\sum_{\{x,y\}\in E}\mu_{xy}\bigl(\psi(x)-\psi(y)\bigr)^2
			\bigl(g_R(x)^s+g_R(y)^s\bigr)  \\
			&\quad +C_s\psi(o)^2g_R(o)^{s-1}.
		\end{aligned}
	\end{equation}
\end{lemma}

\begin{proof}
	Apply Lemma \ref{lem:two-point-caccioppoli} on each edge with
	\[
	a=g_R(x),\quad b=g_R(y),\quad
	\alpha=\psi(x),\quad \beta=\psi(y).
	\]
	Summing over unordered edges gives
	\[
	\begin{aligned}
		c_s\cE(\psi g_R^{s/2},\psi g_R^{s/2})
		&\le
		\cE(g_R,\psi^2g_R^{s-1})  \\
		&\quad+C_s\sum_{\{x,y\}\in E}\mu_{xy}(\psi(x)-\psi(y))^2
		(g_R(x)^s+g_R(y)^s).
	\end{aligned}
	\]
	Since $g_R=0$ on $V\setminus B_R$ and $s>1$, both
	$\psi g_R^{s/2}$ and $\psi^2g_R^{s-1}$ are supported in $B_R$.  Hence
	\eqref{eq:green-energy-identity} gives
	\[
	\cE(g_R,\psi^2g_R^{s-1})=\psi(o)^2g_R(o)^{s-1}.
	\]
	This proves \eqref{eq:green-caccioppoli} after changing the constants.
\end{proof}

\begin{proposition}\label{prop:hardy-caccioppoli-absorption}
	Let $h:V\to(0,\infty)$ and $\lambda:V\to(0,\infty)$ satisfy
	\[
	Lh(x)\ge \lambda(x)h(x)\mu(x),
	\qquad x\in V.
	\]
	Then, for every $s>1$ and $P>0$, there is $C=C(s,P)>0$ such that, for
	all $R\ge1$,
	\begin{equation}\label{eq:linear-absorption}
		\sum_{x\in B_R}\rho(x)^{2P}g_R(x)^s\mu(x)
		\le
		C\sum_{x\in B_R}\lambda(x)^{-P}g_R(x)^s\mu(x).
	\end{equation}
\end{proposition}

\begin{proof}
	Set
	\[
	\psi(x)=\rho(x)^{P+1}=(1+d(o,x))^{P+1},\qquad x\in V.
	\]
	This function is defined on all of $V$, so the same edge estimate applies
	across the Dirichlet boundary of $B_R$.  Apply
	\eqref{eq:ground-state-hardy} to $\phi=\psi g_R^{s/2}$ and then use
	Lemma \ref{lem:green-caccioppoli}.  Since $\rho$ is $1$-Lipschitz on
	edges and $x\sim y$ implies $\rho(x)\asymp\rho(y)$,
	\[
	|\psi(x)-\psi(y)|^2
	\le C_P\bigl(\rho(x)^{2P}+\rho(y)^{2P}\bigr),
	\qquad x\sim y.
	\]
	Consequently,
	\[
	\begin{aligned}
	&\bigl(\psi(x)-\psi(y)\bigr)^2
	\bigl(g_R(x)^s+g_R(y)^s\bigr)  \\
	&\qquad\le C_P\Bigl(
	\rho(x)^{2P}g_R(x)^s+\rho(y)^{2P}g_R(y)^s\Bigr),
	\qquad x\sim y.
	\end{aligned}
	\]
	After summing over unordered edges and using
	$\sum_{y\sim x}\mu_{xy}=\mu(x)$, the error term in
	\eqref{eq:green-caccioppoli} is bounded by
	\[
	C_{s,P}\sum_{x\in B_R}\rho(x)^{2P}g_R(x)^s\mu(x).
	\]
	Since $\psi(o)=1$, we obtain
	\begin{equation}\label{eq:pre-absorption}
		\sum_{x\in B_R}\lambda(x)\rho(x)^{2P+2}g_R(x)^s\mu(x)
		\le
		C\sum_{x\in B_R}\rho(x)^{2P}g_R(x)^s\mu(x)
		+Cg_R(o)^{s-1}.
	\end{equation}

	For every $t\ge0$ and every $\eps>0$, Young's inequality gives
	\begin{equation*}
		t^P\le \eps t^{P+1}+C_P\eps^{-P}.
	\end{equation*}
	Taking $t=\lambda(x)\rho(x)^2$ and multiplying by $\lambda(x)^{-P}$,
	we obtain
	\[
	\rho(x)^{2P}
	\le
	\eps\lambda(x)\rho(x)^{2P+2}
	+C_P\eps^{-P}\lambda(x)^{-P}.
	\]
	Multiplying by $g_R(x)^s\mu(x)$, summing over $B_R$, and using
	\eqref{eq:pre-absorption}, we get
	\[
	\begin{aligned}
	\sum_{x\in B_R}\rho(x)^{2P}g_R(x)^s\mu(x)
	&\le
	\eps C\sum_{x\in B_R}\rho(x)^{2P}g_R(x)^s\mu(x)\\
	&\quad+C\eps g_R(o)^{s-1}
	+C_P\eps^{-P}
	\sum_{x\in B_R}\lambda(x)^{-P}g_R(x)^s\mu(x).
	\end{aligned}
	\]
	Choose $\eps>0$ so small that the first term on the right is absorbed
	into the left.  It remains to control the pole term.  Since $h>0$,
	\[
	\lambda(o)h(o)\mu(o)
	\le Lh(o)
	=\mu(o)h(o)-\sum_{y\sim o}\mu_{oy}h(y)
	<\mu(o)h(o),
	\]
	and therefore $\lambda(o)<1$.  Moreover, since $Lg_R(o)=1$ and
	$g_R\ge0$,
	\[
	1=\mu(o)g_R(o)-\sum_{y\sim o}\mu_{oy}g_R(y)
	\le\mu(o)g_R(o),
	\]
	so $g_R(o)\ge1/\mu(o)$.  Hence
	\[
	\begin{aligned}
	g_R(o)^{s-1}
	&\le \mu(o)g_R(o)^s
	\le \lambda(o)^{-P}\mu(o)g_R(o)^s  \\
	&\le \sum_{x\in B_R}\lambda(x)^{-P}g_R(x)^s\mu(x).
	\end{aligned}
	\]
	This proves \eqref{eq:linear-absorption}.
\end{proof}
\section{A uniform upper bound for the weighted Green mass}\label{upp-bound}
This section combines the weighted Green estimate in Proposition
\ref{prop:hardy-caccioppoli-absorption} with Picone testing for the two
inequalities to show that, assuming the existence of a nontrivial nonnegative solution of the
asymmetric system, $M_{pq,P}(R)$ is bounded uniformly in $R$.
\begin{proposition}\label{prop:system-upper-endpoint}
	Assume $p\ne q$, $pq>1$, and let $(u,v)$ be a nontrivial nonnegative solution of
	\eqref{eq:system-normalized}.  Set
	\[
	P=\max\{p,q\},
	\qquad s=pq.
	\]
	Then
	\begin{equation*}
		\sup_{R\ge1}M_{pq,P}(R)<\infty.
	\end{equation*}
	The upper bound is independent of $R$.  When $q>p$, it depends only on
	$p,q,\mu(o)$, and $v(o)$; when $p>q$, it depends only on
	$p,q,\mu(o)$, and $u(o)$.
\end{proposition}

\begin{proof}
	By Lemma \ref{lem:zero-propagation}, $u,v>0$ everywhere.  First assume
	$q>p$, so $P=q$, and set
	\[
	\lambda=\frac{v^p}{u}.
	\]
	Then the $u$-inequality gives
	\[
	Lu\ge v^p\mu=\lambda u\mu.
	\]
	Applying Proposition \ref{prop:hardy-caccioppoli-absorption} with
	$h=u$, $P=q$, and $s=pq$, we obtain
	\begin{equation*}
		M_{pq,q}(R)
		\le
		C_{p,q}\sum_{x\in B_R}\lambda(x)^{-q}g_R(x)^{pq}\mu(x).
	\end{equation*}
	Moreover,
	\[
	\lambda^{-q}g_R^{pq}
	=\left(\frac{u}{v^p}\right)^qg_R^{pq}
	=u^q\left(\frac{g_R}{v}\right)^{pq}.
	\]
	Since $pq>1$, estimate \eqref{eq:test-v} is applicable with $B=pq$ and gives
	\[
	\sum_{x\in B_R}u(x)^q\left(\frac{g_R(x)}{v(x)}\right)^{pq}\mu(x)
	\le
	\frac{pq}{pq-1}\left(\frac{g_R(o)}{v(o)}\right)^{pq-1}.
	\]
	Therefore
	\begin{equation}\label{eq:q-p-upper-2}
		M_{pq,q}(R)
		\le C_{p,q}v(o)^{1-pq}g_R(o)^{pq-1}.
	\end{equation}
	On the other hand, $\rho(o)=1$, so
	\[
	M_{pq,q}(R)\ge \mu(o)g_R(o)^{pq}.
	\]
	Combining this with \eqref{eq:q-p-upper-2} and dividing by the positive
	quantity $g_R(o)^{pq-1}$ yields
	\[
	g_R(o)\le \frac{C_{p,q}}{\mu(o)v(o)^{pq-1}},
	\]
	uniformly in $R$.  Substitution into \eqref{eq:q-p-upper-2} proves the
	desired uniform bound for $M_{pq,q}(R)$.
	
	Now assume $p>q$, so $P=p$, and set
	\[
	\lambda=\frac{u^q}{v}.
	\]
	The $v$-inequality gives $Lv\ge\lambda v\mu$.  Applying Proposition
	\ref{prop:hardy-caccioppoli-absorption} with $h=v$, $P=p$, and $s=pq$ gives
	\[
	M_{pq,p}(R)
	\le
	C_{p,q}\sum_{x\in B_R}\lambda(x)^{-p}g_R(x)^{pq}\mu(x).
	\]
	Here
	\[
	\lambda^{-p}g_R^{pq}
	=v^p\left(\frac{g_R}{u}\right)^{pq}.
	\]
	Since $pq>1$, estimate \eqref{eq:test-u} with $A=pq$ yields
	\begin{equation}\label{eq:p-q-upper-2}
	M_{pq,p}(R)
	\le C_{p,q}u(o)^{1-pq}g_R(o)^{pq-1}.
	\end{equation}
	Together with $M_{pq,p}(R)\ge\mu(o)g_R(o)^{pq}$, this gives
	\[
	g_R(o)\le \frac{C_{p,q}}{\mu(o)u(o)^{pq-1}},
	\]
	uniformly in $R$.  Substitution into \eqref{eq:p-q-upper-2} completes the proof.
\end{proof}

\section{Weighted lower bound via flow decomposition}\label{lbdd}

The finite Green voltage $g_R$ induces a unit electrical current on the network
obtained by collapsing $V\setminus B_R$ to one boundary vertex.  We orient every
nonzero current edge from higher to lower Green voltage.  The voltage then decreases
strictly along directed edges, so the resulting flow is already acyclic and no cycle
pruning is needed.  The classical flow decomposition theorem represents this current
by a probability measure on directed root-to-boundary paths.  Weighted one-dimensional
Hardy estimates along these paths, chronological first exits, and a parallel-sum
argument then give the lower bound for $M_{s,P}(R)$.

For $k\ge0$, let $b_k$ denote the total weight of the edges crossing the
boundary of $B_k$:
\begin{equation}\label{eq:bk}
	b_k=
	\sum_{\substack{x\in B_k,\ y\notin B_k\\x\sim y}}\mu_{xy}.
\end{equation}
Since the graph is infinite and connected, $b_k>0$.

\begin{proposition}\label{prop:weighted-flow-lower}
	For every $s>1$, every $P\ge0$, and every $R\ge1$,
	\begin{equation}
		M_{s,P}(R)
		\ge
		c_{s,P}\sum_{n=1}^{R}
		n^{2P+1}
		\left(\sum_{k=n}^{R}\frac1{b_k}\right)^{s-1},
	\end{equation}
	where $c_{s,P}>0$ depends only on $s$ and $P$.
\end{proposition}

The flow decomposition theorem for finite acyclic flows used below is classical.  The
weighted estimate in Proposition \ref{prop:weighted-flow-lower}, including its
first-exit and parallel-sum steps, is proved in full in the following three
subsections.

\subsection{Flow decomposition of the Green current}

Collapse $V\setminus B_R$ to one boundary vertex $\partial$.  Conductances inside $B_R$
are unchanged, and the conductance from $x\in B_R$ to $\partial$ is the sum of the
conductances from $x$ to $V\setminus B_R$.  Denote the conductance of an edge $e$
in the collapsed network by $\mu_e^{(R)}$; when $e=\{x,y\}$, we also write
$\mu_{xy}^{(R)}$.  Set $g_R(\partial)=0$.
Orient each edge with nonzero voltage drop
from larger $g_R$ to smaller $g_R$, and discard zero-drop edges.  For an oriented retained
edge $e=(x,y)$, set
\[
\delta_e=g_R(x)-g_R(y)>0,
\qquad
\theta_e=\mu_e^{(R)}\delta_e.
\]
For $x\in B_R$, the divergence of this directed current is
\[
\begin{aligned}
\operatorname{div}\theta(x)
&=\sum_{e=(x,y)}\theta_e-\sum_{e=(y,x)}\theta_e  \\
&=\sum_y\mu_{xy}^{(R)}\bigl(g_R(x)-g_R(y)\bigr)
=Lg_R(x)=\1_{\{o\}}(x).
\end{aligned}
\]
Since the total divergence on the finite collapsed network is zero,
$\operatorname{div}\theta(\partial)=-1$.  Thus the total flow strength is one.
Moreover, $g_R$ decreases strictly along every retained directed edge, so no directed
cycle is possible.  Hence $\theta$ is an acyclic unit flow from $o$ to $\partial$.
The classical flow decomposition theorem for finite acyclic flows therefore yields a
probability measure on finite directed paths
\[
\gamma=(x_0=o,e_0,x_1,e_1,\dots,e_{m-1},x_m=\partial)
\]
such that
\begin{equation}\label{eq:path-probability}
	\mathbb P(\gamma\hbox{ uses }e)=\theta_e
\end{equation}
for every retained edge $e$; see \cite{FFbook}.

For a sampled path, set
\[
\delta_i=g_R(x_i)-g_R(x_{i+1})=\delta_{e_i}>0.
\]

\begin{lemma}\label{lem:weighted-mass-domination}
	For every $s>1$ and $P\ge0$,
	\begin{equation*}
		M_{s,P}(R)
		\ge
		\mathbb E_\gamma\sum_{i=0}^{m-1}
		\rho(x_i)^{2P}\frac{g_R(x_i)^s}{\delta_i}.
	\end{equation*}
\end{lemma}

\begin{proof}
	Using \eqref{eq:path-probability} and $\delta_e=\theta_e/\mu_e^{(R)}$,
	\[
	\begin{aligned}
		\mathbb E_\gamma\sum_i\rho(x_i)^{2P}\frac{g_R(x_i)^s}{\delta_i}
		&=\sum_{e=(x,y)}\theta_e\rho(x)^{2P}\frac{g_R(x)^s}{\delta_e}  \\
		&=\sum_{e=(x,y)}\mu_e^{(R)}\rho(x)^{2P}g_R(x)^s.
	\end{aligned}
	\]
	For a fixed $x\in B_R$, the sum of the conductances of retained outgoing edges from $x$
	is at most the full vertex weight $\mu(x)$.  Therefore the last display is bounded by
	\[
	\sum_{x\in B_R}\rho(x)^{2P}g_R(x)^s\mu(x)=M_{s,P}(R).
	\]
\end{proof}

\subsection{Weighted one-dimensional Hardy estimate}

\begin{lemma}\label{lem:weighted-hardy-copson}
	Let $\alpha\ge0$, $N\ge2$, and $a_1,\dots,a_N>0$.  Put
	$A_m=\sum_{i=1}^m a_i$.  Then
	\begin{equation*}
		\sum_{i=1}^{N}\frac{i^\alpha}{a_i}
		\ge
		c_\alpha\sum_{m=2}^{N}\frac{m^{\alpha+1}}{A_{m-1}}.
	\end{equation*}
\end{lemma}

\begin{proof}
	For each $m\ge2$, let
	\[
	I_m=\{i:\lceil m/2\rceil\le i\le m-1\}.
	\]
	Then $|I_m|\ge c m$ and $i\ge m/2$ for $i\in I_m$.  By Cauchy's inequality,
	\[
	\sum_{i\in I_m}\frac{i^\alpha}{a_i}
	\ge
	c_\alpha m^\alpha\frac{|I_m|^2}{\sum_{i\in I_m}a_i}
	\ge
	c_\alpha\frac{m^{\alpha+2}}{A_{m-1}}.
	\]
	Hence
	\[
	\frac{m^{\alpha+1}}{A_{m-1}}
	\le
	C_\alpha\frac1m\sum_{i\in I_m}\frac{i^\alpha}{a_i}.
	\]
	Summing over $m$ and interchanging sums gives
	\[
	\sum_{m=2}^{N}\frac{m^{\alpha+1}}{A_{m-1}}
	\le
	C_\alpha\sum_{i=1}^{N}\frac{i^\alpha}{a_i}
	\sum_{m:\ i\in I_m}\frac1m.
	\]
	If $i\in I_m$, then $i<m\le2i+2$, so
	$\sum_{m:\ i\in I_m}m^{-1}\le C$.  This proves the lemma.
\end{proof}

\begin{lemma}\label{lem:tail-hardy}
	Let $s>1$, $\alpha\ge0$, and let $\delta_1,\dots,\delta_N>0$.  Define
	\[
	H_m=\sum_{j=m}^{N}\delta_j,
	\qquad 1\le m\le N.
	\]
	Then
	\begin{equation}\label{eq:tail-hardy}
		\sum_{m=1}^{N}m^\alpha\frac{H_m^s}{\delta_m}
		\ge
		c_{s,\alpha}\sum_{m=1}^{N}m^{\alpha+1}H_m^{s-1}.
	\end{equation}
\end{lemma}

\begin{proof}
	The $m=1$ term on the right is controlled by the $m=1$ term on the left because
	$\delta_1\le H_1$.  For $m\ge2$, set
	\[
	a_i=\frac{\delta_i}{H_i^s},
	\qquad
	A_{m-1}=\sum_{i=1}^{m-1}a_i.
	\]
	Since $H_{i+1}=H_i-\delta_i$ and $t^{-s}\ge H_i^{-s}$ for
	$t\in[H_{i+1},H_i]$, we have
	\[
	\frac{\delta_i}{H_i^s}
	\le
	\int_{H_{i+1}}^{H_i}t^{-s}\,dt.
	\]
	Therefore
	\[
	A_{m-1}
	\le
	\int_{H_m}^{H_1}t^{-s}\,dt
	\le
	\frac{H_m^{1-s}}{s-1}.
	\]
	Hence
	\[
	H_m^{s-1}\le \frac{1}{s-1}\frac1{A_{m-1}}.
	\]
	Applying Lemma \ref{lem:weighted-hardy-copson} with
	$\alpha$ and $a_i=\delta_i/H_i^s$ gives
	\[
	\begin{aligned}
		\sum_{i=1}^{N}i^\alpha\frac{H_i^s}{\delta_i}
		=\sum_{i=1}^{N}\frac{i^\alpha}{a_i}
		&\ge
		c_\alpha\sum_{m=2}^{N}\frac{m^{\alpha+1}}{A_{m-1}}  \\
		&\ge
		c_{s,\alpha}\sum_{m=2}^{N}m^{\alpha+1}H_m^{s-1}.
	\end{aligned}
	\]
	Together with the $m=1$ observation, this proves \eqref{eq:tail-hardy}.
\end{proof}

\subsection{First exits and parallel sums}

For a sampled path $\gamma$, let $\alpha_k(\gamma)$ be the first edge of
$\gamma$ crossing from $B_k$ to $V\setminus B_k$, $1\le k\le R$.  Such an
edge exists because the path begins at $o\in B_k$ and terminates at the collapsed
boundary.  Write
\[
\delta_k^*(\gamma)=\delta_{\alpha_k(\gamma)},
\qquad
H_n^*(\gamma)=\sum_{k=n}^{R}\delta_k^*(\gamma).
\]
Thus $\delta_k^*$ is the selected first-exit voltage drop at scale $k$, and
$H_n^*$ is the corresponding tail of selected drops from scale $n$ onward.
Let $z_k$ be the endpoint of $\alpha_k(\gamma)$ lying in $B_k$.
Graph distance changes by at most one across an edge, so
$d(o,z_k)=k$ and $\rho(z_k)=k+1$.

The selected edges are distinct and occur in strict chronological order.  Indeed,
to leave $B_\ell$ with $\ell>k$, the path must first leave $B_k$; moreover, an edge
between adjacent vertices can cross at most one of the metric cuts
$B_j\,|\,V\setminus B_j$.  Since all terms in Lemma
\ref{lem:weighted-mass-domination} are nonnegative, keeping only these selected
edges gives
\[
M_{s,P}(R)\ge \mathbb E_\gamma\sum_{k=1}^R
\rho(z_k)^{2P}\frac{g_R(z_k)^s}{\delta_k^*}.
\]
The part of $\gamma$ from $z_k$ to the boundary contains
$\alpha_k,\ldots,\alpha_R$.  Since $g_R$ decreases along the path and vanishes
at the boundary, $g_R(z_k)$ is the sum of all voltage drops along this terminal
segment.  Hence
\[
g_R(z_k)\ge \sum_{j=k}^R\delta_j^*=H_k^*.
\]
Using $(k+1)^{2P}\ge k^{2P}$, we obtain
\[
M_{s,P}(R)\ge \mathbb E_\gamma\sum_{k=1}^R
k^{2P}\frac{(H_k^*)^s}{\delta_k^*}.
\]
For each fixed sampled path, Lemma \ref{lem:tail-hardy} with $\alpha=2P$
yields
\[
\sum_{k=1}^R k^{2P}\frac{(H_k^*)^s}{\delta_k^*}
\ge c_{s,P}\sum_{n=1}^R n^{2P+1}(H_n^*)^{s-1}.
\]
Taking expectation gives
\begin{equation}\label{eq:selected-edge-lower}
	M_{s,P}(R)
	\ge
	c_{s,P}\sum_{n=1}^{R}
	n^{2P+1}\mathbb E_\gamma\bigl(H_n^*\bigr)^{s-1}.
\end{equation}

We now estimate $\mathbb E_\gamma(H_n^*)^{s-1}$ from below.  Fix $n$.
If a directed edge $e$ is selected as the first exit edge from $B_k$, then the sampled path
uses $e$.  Hence, using \eqref{eq:path-probability},
\[
\begin{aligned}
	\mathbb E_\gamma\frac1{\delta_k^*}
	&=\sum_e\mathbb P(\alpha_k=e)\frac1{\delta_e}\\
	&\le
	\sum_{\substack{e\text{ retained and oriented}\\
			\text{outward across }B_k}}
	\theta_e\frac1{\delta_e}\\
	&=
	\sum_{\substack{e\text{ retained and oriented}\\
			\text{outward across }B_k}}
	\mu_e^{(R)}
	\le b_k.
\end{aligned}
\]
The parallel-sum map
\[
(y_n,\dots,y_R)\mapsto
\left(\sum_{k=n}^{R}\frac1{y_k}\right)^{-1}
\]
is increasing and concave on $(0,\infty)^{R-n+1}$.
Indeed, if $F(y)=(\sum_{k=n}^R y_k^{-1})^{-1}$, then
$\partial_jF=F^2/y_j^2>0$, and Cauchy's inequality gives
\[
\xi^{\mathsf T}D^2F(y)\xi
=2F^3\left(\sum_{k=n}^R\frac{\xi_k}{y_k^2}\right)^2
-2F^2\sum_{k=n}^R\frac{\xi_k^2}{y_k^3}\le0.
\]
Since
\[
(H_n^*)^{-1}
=F\left(\frac1{\delta_n^*},\ldots,\frac1{\delta_R^*}\right),
\]
Jensen's inequality with respect to the path probability measure gives
\[
\mathbb E_\gamma (H_n^*)^{-1}
\le
F\left(\mathbb E_\gamma\frac1{\delta_n^*},\ldots,
\mathbb E_\gamma\frac1{\delta_R^*}\right)
\le
\left(\sum_{k=n}^{R}\frac1{b_k}\right)^{-1}.
\]
The first inequality uses the concavity of the parallel-sum map, and the
second uses its coordinatewise monotonicity together with
$\mathbb E_\gamma(1/\delta_k^*)\le b_k$.

Finally, for every positive random variable $X$ and every $r>0$, H\"older's
inequality gives
\[
1=\mathbb E\!\left[X^{r/(r+1)}X^{-r/(r+1)}\right]
\le (\mathbb E X^r)^{1/(r+1)}
(\mathbb E X^{-1})^{r/(r+1)},
\]
and therefore
\[
\mathbb E X^r\ge (\mathbb E X^{-1})^{-r}.
\]
Applying this to $X=H_n^*$ and $r=s-1$ yields
\begin{equation}\label{eq:first-exit-lower}
	\mathbb E_\gamma (H_n^*)^{s-1}
	\ge
	\left(\sum_{k=n}^{R}\frac1{b_k}\right)^{s-1}.
\end{equation}
Combining \eqref{eq:selected-edge-lower} and \eqref{eq:first-exit-lower}
proves Proposition \ref{prop:weighted-flow-lower}.
\par\hfill\qedsymbol



We need one deterministic comparison between ball volumes and the boundary
weights $b_k$.

\begin{lemma}\label{lem:volume-tail-comparison}
	Let $a\ge0$ and $r>0$.  If
	\begin{equation}\label{eq:volume-series-general}
		\sum_{n=2}^{\infty}\frac{n^{a+2r}}{\mu(B_n)^r}=\infty,
	\end{equation}
	then
	\begin{equation}\label{eq:volume-tail}
		\sum_{n=1}^{\infty}
		n^a\left(\sum_{k=n}^{\infty}\frac{k}{\mu(B_k)}\right)^r
		=\infty.
	\end{equation}
\end{lemma}

\begin{proof}
	It suffices to use even indices.  Since $\mu(B_n)$ is nondecreasing, the odd term
	$(2m+1)^{a+2r}/\mu(B_{2m+1})^r$ is bounded by a constant multiple of
	$(2m)^{a+2r}/\mu(B_{2m})^r$.  Hence \eqref{eq:volume-series-general} implies
	\[
	\sum_{m=1}^{\infty}\frac{(2m)^{a+2r}}{\mu(B_{2m})^r}=\infty.
	\]
	For every $m$,
	\[
	\sum_{k=m}^{2m}\frac{k}{\mu(B_k)}
	\ge
	\frac1{\mu(B_{2m})}\sum_{k=m}^{2m}k
	\ge c\frac{m^2}{\mu(B_{2m})}.
	\]
	Therefore
	\[
	m^a\left(\sum_{k=m}^{\infty}\frac{k}{\mu(B_k)}\right)^r
	\ge
	c\frac{m^{a+2r}}{\mu(B_{2m})^r}.
	\]
	Summing over $m$ proves \eqref{eq:volume-tail}.
\end{proof}

\begin{lemma}\label{lem:cut-resistance-comparison}
	Let $a\ge0$ and $r>0$.  If
	\[
	\sum_{n=1}^{\infty}
	n^a\left(\sum_{k=n}^{\infty}\frac{k}{\mu(B_k)}\right)^r
	=\infty,
	\]
	then
	\begin{equation}\label{eq:cut-series-divergence}
		\sum_{n=1}^{\infty}
		n^a\left(\sum_{k=n}^{\infty}\frac1{b_k}\right)^r
		=\infty.
	\end{equation}
\end{lemma}

\begin{proof}
	For this proof only, write
	\[
	K_n=\sum_{k=n}^{\infty}\frac1{b_k},
	\qquad
	H_n=\sum_{k=n}^{\infty}\frac{k}{\mu(B_k)}.
	\]
	We first prove
	\begin{equation}\label{eq:K-H-shift}
		K_n\ge cH_{2n}
	\end{equation}
	for all $n\ge1$.  For finite $N>2n$, apply Lemma
	\ref{lem:weighted-hardy-copson} with $\alpha=0$ to the sequence
	$b_n,b_{n+1},\dots,b_N$.  After shifting the indices, we obtain
	\[
	\sum_{j=n}^{N}\frac1{b_j}
	\ge c\sum_{m=n}^{N-1}
	\frac{m-n+2}{\sum_{j=n}^{m}b_j}.
	\]
	For $m\ge2n$, $m-n+2\ge m/2$.  To estimate the denominator, let
	\[
	\Pi_j=\bigl\{\{x,y\}\in E:x\in B_j,\ y\notin B_j\bigr\},
	\qquad 0\le j\le m.
	\]
	Because graph distance changes by at most one across an edge, the cutsets
	$\Pi_0,\ldots,\Pi_m$ are pairwise edge-disjoint.  Every edge in their union
	has an endpoint in $B_m$.  Consequently,
	\[
	\sum_{j=n}^{m}b_j
	\le \sum_{j=0}^{m}b_j
	=\sum_{e\in\bigcup_{j=0}^{m}\Pi_j}\mu_e
	\le \sum_{x\in B_m}\sum_{y\sim x}\mu_{xy}
	=\mu(B_m).
	\]
	Hence
	\[
	\sum_{j=n}^{N}\frac1{b_j}
	\ge c\sum_{m=2n}^{N-1}\frac{m}{\mu(B_m)}.
	\]
	Letting $N\to\infty$ proves \eqref{eq:K-H-shift}.
	
	Since $H_n$ is decreasing, the divergence of
	$\sum n^aH_n^r$ implies the divergence of its even subseries.  Indeed,
	\[
	(2m-1)^aH_{2m-1}^r\le C_a(2m-2)^aH_{2m-2}^r
	\]
	for $m\ge2$.  Therefore
	\[
	\sum_{n=1}^{\infty}n^aH_{2n}^r=\infty.
	\]
	Using \eqref{eq:K-H-shift},
	\[
	\sum_{n=1}^{\infty}n^aK_n^r
	\ge c\sum_{n=1}^{\infty}n^aH_{2n}^r
	=\infty.
	\]
	This proves \eqref{eq:cut-series-divergence}.
\end{proof}

\section{Proofs of the main results}\label{prf}

\begin{proof}[Proof of Theorem \ref{thm:sharp-asymmetric}]
	Suppose, to the contrary, that there is a nontrivial nonnegative solution $(u,v)$.  By
	Proposition \ref{prop:system-upper-endpoint}, with
	\[
	s=pq,
	\qquad
	P=\max\{p,q\},
	\]
	we have
	\begin{equation}\label{eq:bounded-M}
		\sup_{R\ge1}M_{pq,P}(R)<\infty.
	\end{equation}
	
	Because $pq-1>0$, the finite sums satisfy
	\[
	\sum_{n=1}^{R}n^{2P+1}
	\left(\sum_{k=n}^{R}\frac1{b_k}\right)^{pq-1}
	\uparrow
	\sum_{n=1}^{\infty}n^{2P+1}
	\left(\sum_{k=n}^{\infty}\frac1{b_k}\right)^{pq-1}.
	\]
	Proposition \ref{prop:weighted-flow-lower} bounds $M_{pq,P}(R)$ from
	below by a fixed positive multiple of the finite sum on the left.  Hence
	divergence of the limiting series contradicts \eqref{eq:bounded-M}.
	
	It remains to derive this divergence from the volume condition
	\eqref{eq:sharp-series}.  Apply Lemmas \ref{lem:volume-tail-comparison} and
	\ref{lem:cut-resistance-comparison} with $a=2P+1$ and $r=pq-1$.
	The resulting volume exponent is
	\[
	(2P+1)+2(pq-1)=2pq+2P-1,
	\]
	exactly the exponent in \eqref{eq:sharp-series}.  Therefore
	\[
	\sum_{n=1}^{\infty}n^{2P+1}
	\left(\sum_{k=n}^{\infty}\frac1{b_k}\right)^{pq-1}
	=\infty,
	\]
	which is the required contradiction.  Thus no nontrivial nonnegative
	solution exists, and the only nonnegative solution is
	$u\equiv v\equiv0$.
\end{proof}

\begin{proof}[Proof of Corollary \ref{cor:critical-log}]
	Choose an integer $R_0\ge3$ such that the assumed volume bound holds for
	all integers $R\ge R_0$.  Since
	\[
	(pq-1)D_{\rm sys}
	=2(pq-1)+2(P+1)=2pq+2P,
	\]
	we have, for $R\ge R_0$,
	\[
	\begin{aligned}
		\frac{R^{2pq+2P-1}}{\mu(B_R)^{pq-1}}
		&\ge
		c\frac{R^{2pq+2P-1}}{
			R^{(pq-1)D_{\rm sys}}(\log R)^{\theta(pq-1)}}  \\
		&=c\frac{1}{R(\log R)^{\theta(pq-1)}}.
	\end{aligned}
	\]
	The finite initial segment is irrelevant, and the logarithmic harmonic
	series on the right diverges exactly when
	$\theta(pq-1)\le1$.  Hence \eqref{eq:sharp-series} holds, and Theorem
	\ref{thm:sharp-asymmetric} gives $u\equiv v\equiv0$.
\end{proof}

\begin{proof}[Proof of Theorem \ref{thm:symmetric}]
	Let $(u,v)$ be a nonnegative solution with $p=q=\sigma>1$, and put
	$w=u+v$.  By convexity,
	\[
	-\Delta w\ge u^\sigma+v^\sigma
	\ge 2^{1-\sigma}w^\sigma.
	\]
	Consequently,
	\[
	-\Delta(w/2)\ge (w/2)^\sigma.
	\]
	By the sharp scalar volume criterion proved in \cite{GHHS26}, condition
	\eqref{eq:scalar-series} forces $w/2\equiv0$.  Hence $w\equiv0$, and
	the nonnegativity of $u$ and $v$ gives
	$u\equiv v\equiv0$.
\end{proof}

\section{Logarithmic sharpness on weighted half-lines}\label{counter}

We finish with an explicit model showing that the logarithmic endpoint in
Corollary \ref{cor:critical-log} cannot be improved within the critical
polynomial--logarithmic volume scale.

\begin{theorem}\label{thm:weighted-line-sharpness}
	Assume $p,q>0$, $pq>1$, and $p\ne q$.  Let
	\[
	P=\max\{p,q\},
	\qquad
	D_{\rm sys}=2+\frac{2(P+1)}{pq-1}.
	\]
	For every $\eps>0$ there exists an infinite, connected, locally finite
	weighted graph with root $o$ such that
	\[
	\mu(B(o,R))\asymp
	R^{D_{\rm sys}}(\log R)^{1/(pq-1)+\eps}
	\qquad (R\to\infty),
	\]
	and the Lane--Emden system \eqref{eq:system-normalized} has a positive
	solution.
\end{theorem}

We use a weighted half-line.  Let
\[
V=\N_0=\{0,1,2,\dots\},
\qquad
E=\{\{n,n+1\}:n\ge0\},
\]
and assign conductance $\omega_n>0$ to the edge $\{n,n+1\}$.  The associated
vertex measure is
\[
\mu(0)=\omega_0,
\qquad
\mu(n)=\omega_{n-1}+\omega_n\quad(n\ge1).
\]
For a real-valued function $F$ on $\N_0$,
recall that 
\[
LF(0)=\omega_0\bigl(F(0)-F(1)\bigr)
\] 
and for $n\ge1$,
\[
LF(n)=\omega_n\bigl(F(n)-F(n+1)\bigr)+
\omega_{n-1}\bigl(F(n)-F(n-1)\bigr).
\]

The next lemma computes the ball volume for the conductances used below.

\begin{lemma}\label{lem:line-volume}
	Let $D>0$, $\theta\in\mathbb R$, and $N\ge3$, and set
	\[
	\omega_n=(n+N)^{D-1}\bigl(\log(n+N)\bigr)^\theta.
	\]
	Then, with root $o=0$,
	\[
	\mu(B(o,R))\asymp R^D(\log R)^\theta
	\qquad (R\to\infty).
	\]
	The comparison constants may depend on $D$, $\theta$, and $N$.
\end{lemma}

\begin{proof}
	Since $B(o,R)=\{0,1,\dots,R\}$ for integer $R\ge0$,
	\[
	\mu(B(o,R))
	=\omega_0+\sum_{n=1}^{R}(\omega_{n-1}+\omega_n)
	\asymp
	\sum_{n=0}^{R}(n+N)^{D-1}(\log(n+N))^\theta.
	\]
	The summand is eventually monotone, and its variation on every interval
	$[n+N,n+N+1]$ is bounded by a fixed factor.  Hence the last sum is
	comparable, up to a finite initial contribution, to
	\[
	\int_N^{R+N+1}t^{D-1}(\log t)^\theta\,\dd t.
	\]
	Since $D>0$, l'H\^{o}pital's rule gives
	\[
	\int_N^T t^{D-1}(\log t)^\theta\,\dd t
	\sim \frac1D T^D(\log T)^\theta
	\qquad(T\to\infty).
	\]
	The conclusion follows because $R+N+1\asymp R$ for large $R$.
\end{proof}

We also need a quantitative lower estimate for power--logarithmic profiles.

\begin{lemma}\label{lem:line-radial}
	Let $D>2$, $\theta,\delta\in\mathbb R$, and $\gamma>0$.  For all
	sufficiently large integers $N$, set
	\[
	\omega_n=(n+N)^{D-1}\bigl(\log(n+N)\bigr)^\theta,
	\qquad
	F(n)=(n+N)^{-\gamma}\bigl(\log(n+N)\bigr)^{-\delta}.
	\]
	Then the following estimates hold with a constant $c>0$ independent of
	$n$.
	\begin{enumerate}[label=\textup{(\roman*)},leftmargin=2em]
		\item If $0<\gamma<D-2$, then
		\[
		LF(n)\ge
		c\,\mu(n)(n+N)^{-\gamma-2}
		\bigl(\log(n+N)\bigr)^{-\delta}
		\qquad(n\ge0).
		\]
		\item If $\gamma=D-2$ and $\delta<\theta$, then
		\[
		LF(n)\ge
		c\,\mu(n)(n+N)^{-D}
		\bigl(\log(n+N)\bigr)^{-\delta-1}
		\qquad(n\ge0).
		\]
	\end{enumerate}
	Here the threshold for $N$ and the constant $c$ may depend on
	$D,\theta,\gamma$, and $\delta$.
\end{lemma}

\begin{proof}
	Put
	\[
	r=n+N,
	\qquad
	\ell=\log r,
	\qquad
	A(r)=r^{D-1}\ell^\theta,
	\qquad
	f(r)=r^{-\gamma}\ell^{-\delta}.
	\]
	For $n\ge1$, the half-line formula gives
	\[
	LF(n)=A(r)\bigl(f(r)-f(r+1)\bigr)
	+A(r-1)\bigl(f(r)-f(r-1)\bigr).
	\]
	For fixed real exponents $\eta,\kappa$ and $j\le3$, differentiation gives
	\[
	\left|\frac{\dd^j}{\dd r^j}
	\bigl(r^\eta(\log r)^\kappa\bigr)\right|
	\le C_{\eta,\kappa,j}
	r^{\eta-j}(\log r)^\kappa
	\]
	for all sufficiently large $r$.  Taylor's formula with remainder, applied
	at $r$, therefore yields for every $n\ge1$
	\begin{equation}\label{eq:line-discrete-expansion}
	LF(n)=-(Af')'(r)+E(r),
	\qquad
	|E(r)|\le C r^{D-\gamma-4}\ell^{\theta-\delta},
	\end{equation}
	uniformly for large $r$.  Indeed, the error terms are bounded by constant
	multiples of $A|f'''|$, $|A'||f''|$, and $|A''||f'|$.

	Set
	\[
	m=D-\gamma-2,
	\qquad
	a_0=\theta-\delta.
	\]
	Since
	\[
	Af'(r)=-r^m\bigl(\gamma\ell^{a_0}
	+\delta\ell^{a_0-1}\bigr),
	\]
	we have the exact identity
	\begin{equation}\label{eq:line-continuous-main}
	\begin{aligned}
	-(Af')'(r)
	=r^{m-1}\bigl[&\gamma m\ell^{a_0}
	+(\delta m+\gamma a_0)\ell^{a_0-1}\\
	&+\delta(a_0-1)\ell^{a_0-2}\bigr].
	\end{aligned}
	\end{equation}

	Suppose first that $0<\gamma<D-2$, so $m>0$.  After increasing the lower
	threshold for $r$, the two lower logarithmic powers in
	\eqref{eq:line-continuous-main} have absolute value at most
	$\frac14\gamma m\ell^{a_0}$, and hence
	\[
	-(Af')'(r)\ge
	\frac34\gamma m r^{m-1}\ell^{a_0}.
	\]
	The error in \eqref{eq:line-discrete-expansion} is smaller by a factor
	$O(r^{-1})$ and can be bounded by
	$\frac14\gamma m r^{m-1}\ell^{a_0}$.  Thus
	\[
	LF(n)\ge
	\frac12\gamma m r^{D-\gamma-3}\ell^{\theta-\delta}
	\qquad(n\ge1).
	\]

	Now suppose that $\gamma=D-2$ and $\delta<\theta$.  Then $m=0$ and
	$a_0=\theta-\delta>0$.  Formula
	\eqref{eq:line-continuous-main} becomes
	\[
	-(Af')'(r)
	=r^{-1}\bigl[
	\gamma a_0\ell^{a_0-1}
	+\delta(a_0-1)\ell^{a_0-2}
	\bigr].
	\]
	For all sufficiently large $r$, the second term has absolute value at most
	$\frac14\gamma a_0\ell^{a_0-1}$, while the error in
	\eqref{eq:line-discrete-expansion} is bounded by
	$C r^{-2}\ell^{a_0}$ and hence, after increasing the threshold once more,
	by $\frac14\gamma a_0 r^{-1}\ell^{a_0-1}$.  Consequently,
	\[
	LF(n)\ge
	\frac12\gamma a_0 r^{-1}\ell^{a_0-1}
	\qquad(n\ge1).
	\]

	For large $N$ one has, uniformly for $n\ge1$,
	\[
	\mu(n)=A(r-1)+A(r)\asymp A(r)
	=r^{D-1}\ell^\theta.
	\]
	The preceding two estimates therefore give the asserted right-hand sides
	for all $n\ge1$.

	It remains to treat $n=0$.  By the mean value theorem, for some
	$\xi\in(N,N+1)$,
	\[
	F(0)-F(1)=-f'(\xi)
	=\xi^{-\gamma-1}(\log\xi)^{-\delta}
	\left(\gamma+\frac{\delta}{\log\xi}\right).
	\]
	After increasing $N$, the last factor is at least $\gamma/2$, and the
	remaining factors are comparable to
	$N^{-\gamma-1}(\log N)^{-\delta}$.  Since $\mu(0)=\omega_0$, this yields
	\[
	LF(0)\ge
	c\,\mu(0)N^{-\gamma-1}(\log N)^{-\delta}.
	\]
	This is stronger than the required estimate in (i).  When
	$\gamma=D-2$, it is also stronger than the required estimate in (ii),
	because $N\log N\ge1$.  The proof is complete.
\end{proof}

\begin{proof}[Proof of Theorem \ref{thm:weighted-line-sharpness}]
	Define the polynomial decay exponents
	\[
	\alpha=\frac{2(p+1)}{pq-1},
	\qquad
	\beta=\frac{2(q+1)}{pq-1}.
	\]
	They satisfy
	\[
	q\alpha=\beta+2,
	\qquad
	p\beta=\alpha+2.
	\]
	Also set
	\[
	\theta=\frac1{pq-1}+\eps.
	\]

	Assume first that $q>p$.  Then $\beta>\alpha$ and
	$D_{\rm sys}=2+\beta$.  Choose the logarithmic exponents explicitly by
	\[
	b=\frac1{pq-1}+\frac\eps2,
	\qquad
	a=\frac12\left(\frac{b+1}{q}+pb\right).
	\]
	Since $b(pq-1)>1$, the interval
	$((b+1)/q,pb)$ is nonempty, and therefore
	\[
	\frac1{pq-1}<b<\theta,
	\qquad
	qa>b+1,
	\qquad
	a<pb.
	\]
	Apply Lemma \ref{lem:line-radial} with $D=D_{\rm sys}$, and choose
	$N$ sufficiently large that it applies simultaneously to the parameter
	pairs $(\gamma,\delta)=(\alpha,a)$ and $(\beta,b)$.  Set
	\[
	\omega_n=(n+N)^{D_{\rm sys}-1}(\log(n+N))^\theta
	\]
	and define
	\[
	U(n)=(n+N)^{-\alpha}(\log(n+N))^{-a},
	\qquad
	V(n)=(n+N)^{-\beta}(\log(n+N))^{-b}.
	\]
	Because $\alpha<D_{\rm sys}-2=\beta$ and $b<\theta$, the two parts of Lemma
	\ref{lem:line-radial} give constants $c_U,c_V>0$ such that, for every
	$n\ge0$,
	\[
	\begin{aligned}
	LU(n)
	&\ge c_U\mu(n)(n+N)^{-\alpha-2}(\log(n+N))^{-a}
	\ge c_U\mu(n)V(n)^p,\\
	LV(n)
	&\ge c_V\mu(n)(n+N)^{-\beta-2}(\log(n+N))^{-b-1}
	\ge c_V\mu(n)U(n)^q.
	\end{aligned}
	\]
	Here the last inequalities use $p\beta=\alpha+2$, $a<pb$,
	$q\alpha=\beta+2$, and $qa>b+1$.

	Now assume that $p>q$.  Then $\alpha>\beta$ and
	$D_{\rm sys}=2+\alpha$.  Choose
	\[
	a=\frac1{pq-1}+\frac\eps2,
	\qquad
	b=\frac12\left(\frac{a+1}{p}+qa\right).
	\]
	Since $a(pq-1)>1$, the interval $((a+1)/p,qa)$ is nonempty, and hence
	\[
	\frac1{pq-1}<a<\theta,
	\qquad
	pb>a+1,
	\qquad
	b<qa.
	\]
	Apply Lemma \ref{lem:line-radial} with $D=D_{\rm sys}$, choose $N$
	sufficiently large for both profiles, and use the same conductances and
	profiles as above.  This time
	$\beta<D_{\rm sys}-2=\alpha$ and $a<\theta$, so Lemma
	\ref{lem:line-radial} gives constants $c_U,c_V>0$ such that, for every
	$n\ge0$,
	\[
	\begin{aligned}
	LU(n)
	&\ge c_U\mu(n)(n+N)^{-\alpha-2}(\log(n+N))^{-a-1}
	\ge c_U\mu(n)V(n)^p,\\
	LV(n)
	&\ge c_V\mu(n)(n+N)^{-\beta-2}(\log(n+N))^{-b}
	\ge c_V\mu(n)U(n)^q.
	\end{aligned}
	\]
	The last inequalities now use $p\beta=\alpha+2$, $pb>a+1$,
	$q\alpha=\beta+2$, and $b<qa$.

	In either case, choose
	\[
	0<A\le (c_Uc_V^p)^{1/(pq-1)},
	\qquad
	B=\frac{A^q}{c_V},
	\]
	and set $u=AU$, $v=BV$.  Then
	\[
	A^q=c_VB,
	\qquad
	B^p=\frac{A^{pq}}{c_V^p}\le c_UA.
	\]
	Consequently, at every vertex, including $n=0$,
	\[
	Lu=A LU\ge B^p\mu V^p=\mu v^p,
	\qquad
	Lv=B LV\ge A^q\mu U^q=\mu u^q.
	\]
	Thus $(u,v)$ is a positive solution of
	\eqref{eq:system-normalized}.  Finally, Lemma \ref{lem:line-volume} gives
	\[
	\mu(B(o,R))\asymp
	R^{D_{\rm sys}}(\log R)^{1/(pq-1)+\eps}.
	\]
\end{proof}

For the graph in Theorem \ref{thm:weighted-line-sharpness}, the series in
Theorem \ref{thm:sharp-asymmetric} converges.  Indeed,
\[
D_{\rm sys}(pq-1)
=2(pq-1)+2(P+1)
=2pq+2P,
\]
and therefore
\[
\frac{R^{2pq+2P-1}}{\mu(B_R)^{pq-1}}
\asymp
\frac1{R(\log R)^{\theta(pq-1)}}.
\]
Here $\theta=1/(pq-1)+\eps$, so
$\theta(pq-1)=1+\eps(pq-1)>1$ and the last series is summable.  Thus, for
every $\eps>0$, a positive solution exists at logarithmic volume exponent
$1/(pq-1)+\eps$.  This proves that $1/(pq-1)$ is the optimal universal
nonexistence endpoint within the critical polynomial--logarithmic weighted
half-line family.  It does not assert that convergence of
\eqref{eq:sharp-series} implies existence on every weighted graph.


\medskip

\begin{thebibliography}{99}
	
	\bibitem{BermanKonsowa90}
	K. A. Berman and M. H. Konsowa,
	\emph{Random paths and cuts, electrical networks, and reversible Markov chains},
	SIAM J. Discrete Math. \textbf{3} (1990), 311--319.
	
	
	
	
	\bibitem{BuscaManasevich02}
	J. Busca and R. Man\'asevich,
	\emph{A Liouville-type theorem for Lane--Emden systems},
	Indiana Univ. Math. J. \textbf{51} (2002), 37--51.
	
	
	\bibitem{DuongMinh25}
	A. T. Duong and N. C. Minh,
	\emph{Liouville-type theorems for a system of elliptic inequalities on weighted graphs},
	Z. Anal. Anwend. \textbf{44} (2025), no.~3/4, 307--321.
	
	\bibitem{DLS26}
	A. T. Duong, Y. Liu, N. C. Minh, D. T. Quyet, and Y. Sun,
	\emph{Liouville type results for quasilinear elliptic inequalities involving gradient terms on weighted graphs},
	preprint, arXiv:2604.21145 [math.AP], 2026.
	
	\bibitem{FFbook}
	L. R. Ford, Jr., and D. R. Fulkerson,
	\emph{Flows in Networks}, Princeton Landmarks in Mathematics,
	Princeton University Press, Princeton, NJ, 2010.
	
	\bibitem{GW25}
	Y. Ge and L. Wang,
	\emph{$p$-Laplace elliptic inequalities on the graph},
	Commun. Pure Appl. Anal. \textbf{24} (2025), no.~3, 389--411.
	
	\bibitem{GidasSpruck81}
	B. Gidas and J. Spruck,
	\emph{Global and local behavior of positive solutions of nonlinear elliptic equations},
	Comm. Pure Appl. Math. \textbf{34} (1981), 525--598.
	
	
	
	\bibitem{GrigorYannParabolic}
	A. Grigor'yan,
	\emph{Analytic and geometric background of recurrence and non-explosion of Brownian motion on Riemannian manifolds},
	Bull. Amer. Math. Soc. \textbf{36} (1999), 135--249.
	
	\bibitem{GrigorYannStoch}
	A. Grigor'yan,
	\emph{Heat kernel and analysis on manifolds},
	AMS/IP Studies in Advanced Mathematics, vol.~47,
	American Mathematical Society, Providence, RI, 2009.
	
	\bibitem{AGbook}
	A. Grigor'yan,
	\emph{Introduction to Analysis on Graphs}, University Lecture Series, vol.~71,
	American Mathematical Society, Providence, RI, 2018.
	
	
	\bibitem{GS14}
	A. Grigor'yan and Y. Sun,
	\emph{On nonnegative solutions of the inequality $\Delta u+u^{\sigma}\le 0$ on Riemannian manifolds},
	Comm. Pure Appl. Math. \textbf{67} (2014), 1336--1352.
	
	\bibitem{GSV20}
	A. Grigor'yan, Y. Sun, and I. E. Verbitsky,
	\emph{Superlinear elliptic inequalities on manifolds},
	J. Funct. Anal. \textbf{278} (2020), no.~9, Paper No.~108444, 34 pp.
	
	\bibitem{GHS23}
	Q. Gu, X. Huang, and Y. Sun,
	\emph{Semi-linear elliptic inequalities on weighted graphs},
	Calc. Var. Partial Differential Equations \textbf{62} (2023), no.~2,
	Paper No.~42, 14 pp.
	
	\bibitem{GHHS26}
	Q. Gu, L. Hao, X. Huang, and Y. Sun,
	\emph{Flow decomposition, Green testing, and Lane--Emden inequalities on weighted graphs},
	preprint, arXiv:2604.24932 [math.AP], 2026.
	
	\bibitem{HS23}
	L. Hao and Y. Sun,
	\emph{Sharp Liouville type results for semilinear elliptic inequalities involving gradient terms on weighted graphs},
	Discrete Contin. Dyn. Syst. Ser. S \textbf{16} (2023), no.~6, 1484--1516.
	
	\bibitem{LyonsPeres17}
	R. Lyons and Y. Peres,
	\emph{Probability on Trees and Networks},
	Cambridge University Press, Cambridge, 2017.
	
	\bibitem{MinhDuongNguyen24}
	N. C. Minh, A. T. Duong, and N. H. Nguyen,
	\emph{Liouville type theorem for a system of elliptic inequalities on weighted graphs without $(p_0)$-condition},
	Math. Slovaca \textbf{74} (2024), no.~5, 1255--1266.
	
	
	\bibitem{MinhQuyetDuong25}
	N. C. Minh, D. T. Quyet, and A. T. Duong,
	\emph{Liouville-type theorems for systems of elliptic inequalities involving $p$-Laplace operator on weighted graphs},
	Commun. Pure Appl. Anal. \textbf{24} (2025), no.~4, 641--660.
	
	
	\bibitem{Mitidieri96}
	E. Mitidieri,
	\emph{Nonexistence of positive solutions of semilinear elliptic systems in $\mathbb R^N$},
	Differential Integral Equations \textbf{9} (1996), 465--479.
	
	
	\bibitem{MitidieriPokhozhaev01}
	E. Mitidieri and S. I. Pokhozhaev,
	\emph{A priori estimates and blow-up of solutions to nonlinear partial differential equations and inequalities},
	Proc. Steklov Inst. Math. \textbf{234} (2001), 1--362.
	
	
	\bibitem{NashWilliams59}
	C. St. J. A. Nash-Williams,
	\emph{Random walk and electric currents in networks},
	Proc. Cambridge Philos. Soc. \textbf{55} (1959), 181--194.
	
	
	
	\bibitem{PQS07}
	P. Pol\'a\v{c}ik, P. Quittner, and P. Souplet,
	\emph{Singularity and decay estimates in superlinear problems via Liouville-type theorems},
	Duke Math. J. \textbf{139} (2007), 555--579.
	
	
	\bibitem{SC95}
	L. Saloff-Coste,
	\emph{Inequalities for $p$-superharmonic functions on networks},
	Rend. Sem. Mat. Fis. Milano \textbf{65} (1995), 139--158.
	
	\bibitem{SC97}
	L. Saloff-Coste,
	\emph{Some inequalities for superharmonic functions on graphs},
	Potential Anal. \textbf{6} (1997), no.~2, 163--181.
	
	
	
	\bibitem{SerrinZou96}
	J. Serrin and H. Zou,
	\emph{Non-existence of positive solutions of Lane--Emden systems},
	Differential Integral Equations \textbf{9} (1996), 635--653.
	
	\bibitem{Soardi94}
	P. M. Soardi,
	\emph{Potential Theory on Infinite Networks},
	Lecture Notes in Mathematics, vol.~1590, Springer-Verlag, Berlin, 1994.
	
	
	\bibitem{Souplet09}
	P. Souplet,
	\emph{The proof of the Lane--Emden conjecture in four space dimensions},
	Adv. Math. \textbf{221} (2009), 1409--1427.
	
	\bibitem{Sun16}
	Y. Sun,
	\emph{Uniqueness result on nonnegative solutions of a large class of differential inequalities on Riemannian manifolds},
	Pacific J. Math. \textbf{280} (2016), no.~1, 241--254.
	
	\bibitem{WWbook}
	W. Woess,
	\emph{Random Walks on Infinite Graphs and Groups}, Cambridge Tracts in Mathematics, vol.~138,
	Cambridge University Press, Cambridge, 2000.
	
	
	
	\bibitem{XuWangSun18}
	F. Xu, L. Wang, and Y. Sun,
	\emph{Liouville type theorems for systems of elliptic differential inequalities on Riemannian manifolds},
	J. Math. Anal. Appl. \textbf{466} (2018), 426--446.
\end{thebibliography}
\end{document}